\documentclass[11pt]{amsart}

\usepackage[latin1]{inputenc}
\usepackage{amsfonts,amssymb,amsmath,amsthm,enumerate,latexsym,color}
\usepackage{graphicx,url,comment}
\usepackage{float}

\newtheorem{thm}{Theorem}[section]

\newtheorem{rem}{Remark}[section]

\newtheorem{prop}{Proposition}[section]

\newcommand{\R}{\mathbb{R}}

\newcommand{\eps}{\varepsilon}
\newcommand{\p}{\partial}

\title{Opinion formation process in a hierarchical society}

\begin{document}

\author[M.C. Romero Longar, N.Saintier and A. Silva]
{Maria Celeste Romero Longar, Nicolas Saintier and Analia Silva}

\begin{abstract}
In this work we study the formation of consensus in a hierarchical population. We derive the
corresponding kinetic equations, and analyze the long time behaviour of their solutions  for the case of finite number of hierarchical 
obtaining explicit formula for the consensus opinion. 
\end{abstract}

\maketitle

\section{Introduction}

In the last   years an increasing amount of work has been devoted to the mathematical study of models coming from social and economic sciences. 
Among these, an interesting subject is the modelling of opinion formation process where one tries to understand how exchanges of opinions between individuals can result in such dramatic effect as the emergence of consensus, bipolarization, extremism, ... 

Understanding how individuals are affected by others and change their opinion as a result is a well studied question in sociology. Among the social theories that study theses issues are the social impact and social pressure theory \cite{Asch,cial,Lata},  where
agents modify their opinions trying to fit in some social group, and the persuasive
argument theory \cite{vinokur,MasFlache}, where the new opinion appears after an interchange of
arguments among agents. Economists like T. Schelling \cite{Schelling} also studied how microscopic changes in the behaviours of individuals can propagate in a society to become observable at a large scale.  

More recently physicists started taking advantages of the similarities between these questions  and those of statistical mechanics. 
Indeed we can develop strong analogies between statistical mechanics on one hand and sociology, economy on the other hand 
assimilating individuals exchanging opinion (or any other socio-economic quantity) 
during an encounter with particles exchanging energy during shocks. 
Despite being a reductive approach of the complexity of human behaviours, simple models developped by physicists were able to reproduce empirical data. 
These successes led to the creation of two very active field of research, namely sociophysic and econophysic (see e.g. \cite{CFL, CCCC, Slanina}). 
Mathematicians are also actively investigated questions related to the modelling of socio-economic human behaviours mainly using 
active particles and tools from the kinetic theory of gases like Boltzmann-like and Fokker-Planck equations (see \cite{Bell, Bell2, PT}).  

The relative simplicity of these models is useful to assess the impact of some particular aspect of the dynamic like e.g. 
 the sociological assumptions  embodied in the interaction rule modelling the exchange of opinions between individuals, 
the heterogeneity of the individuals or of the opinion, the network of relations between individuals, .... 
The dynamic of the models is usually investigated though intensive agent-based numerical simulations. 
Theoretical results can then be obtained in some cases using tools from statistical mechanics, probability and partial differential equations
(see e.g. \cite{PT, ppss, pps, pps2}). 

Interactions among individuals are usually assumed to be symmetric in the sense that individual  A influences individual B in the same way that B influences A 
(up to particular characteristics of A and B). 
On the other hand many social structures in human society are hierarchized in the sense that members are ranked differently 
(family, work, army, government), and the rank affects the flow of information in the group resulting in asymmetric interactions between individuals. 
There are very few works concerning opinion dynamic in presence of asymmetry between individuals (see \cite{Laguna} and references therein).

To assess the impact of asymmetry among individuals, we consider in this work a hierarchized society where the ranking of individuals directly impact the flow of information  between them. 
We will suppose that an individual convinces automatically any lower ranked individuals and can also convinces a higher ranked individual with a certain probability $p$. 
Thus when $p=0$ information flows only from top-ranked individuals down to the lowest ranked ones. 
On the contrary when $p>0$ this perfect top-to-bottom flow is perturbed by an upward flow of information. 
We are mainly interested in analyzing how the formation of consensus is perturbed by such a noise. 

To do so we consider a large population of individuals each characterized by an opinion and a rank. 
Individuals interact by pair and modify their opinion following a standard tendency to compromise rule but also taking into account their respective ranks. 
Numerical simulations suggest consensus is always reached  in the sense that individuals tend to share the same opinion up to small random fluctuations. 
To gain further insights we formulate an equation satisfied by the distribution of pairs (opinion, rank), 
and analyze its long-time behaviour when there are finitely many different ranks in the population. 
We also consider another source of perturbation  allowing for stubborn individuals i.e. individuals who always keep the same opinion. 
Such individuals have a strong impact on the formation of consensus since they accelerate the formation of consensus and 
drive the opinion of the non-stubborn individuals towards a the mean of their opinion 
(see \cite{ppss} and the references therein). 
As a result of our theoretical analysis we prove that consensus is indeed reached and obtain explicit formula for the limit opinion
taking into account individuals'ranks and stubborness. 
These results are in perfect agreement with the agent-based simulations.

\subsection*{Notations}

We fix some notations and recall some known facts that will be used throughout this paper.

We let $K=[-1,1]\times [0,1]$ and denote a generic element of $K$ by $\varpi:=(w,h)$.

The convex set of Borel probability measures over $K$ is denoted $P(K)$.
The integration of a measurable function $\phi:K\to \R$ with respect to a probability measure $f\in P(K)$ is denoted by $\int_K \phi(\varpi)\,f(d\varpi)$. When $f=f(t)$ depends on time, we denote the integral $\int_K \phi(\varpi)\,f(t,d\varpi)$.
The marginals of $f\in P(K)$ in the $w$- and $h$-variables will be denoted by $f(dw)$ and $f(dh)$ respectively. Thus for any $\phi:K\to \R$ depending only on $w$,
$$ \int_K \phi(\varpi)\,f(d\varpi) = \int_{-1}^1 \phi(w)\,f(dw), $$
and likewise if $\phi$ depends only on $h$,
$$ \int_K \phi(\varpi)\,f(d\varpi) = \int_0^1 \phi(h)\,f(dh). $$

Two notions of convergence are useful on $P(K)$: the convergence with respect to the Total Variation (TV) norm, and the weak convergence.
The TV norm of $f\in P(K)$ is
\begin{equation}\label{TVnorm}
 \|f\|_{TV}= \sup \left\lbrace \int_K \phi(\varpi)\,f(d\varpi):\, \phi \in C(K),\, \|\phi\|_\infty \leq 1 \right\rbrace.
\end{equation}
Then $(P(K), \|.\|_{TV})$ is a Banach space. However the TV norm is too rigid for our purpose. 
For instance even if $y\to x$, we do not have
$\lim_{y\to x}\|\delta_y-\delta_x\|_{TV}=0$ - in fact $\|\delta_y-\delta_x\|_{TV}=2$ if $x\neq y$.
Moreover it is difficult to obtain compactness in the TV norm.
The weak convergence is a weaker notion of convergence which turns to be out more useful for us.
We say that a sequence $(f_n)_n\subset P(K)$ weakly converges to $f\in P(K)$  if
\begin{equation}\label{WeakConv}
\int_K \phi(\varpi)\,f_n(d\varpi)\to \int_K \phi(\varpi)\,f(d\varpi) \qquad \text{for any $\phi\in C(K)$.}
\end{equation}
Since $K$ is compact, Prokhorov's Theorem gives that $P(K)$ is compact for the weak convergence.
Moreover the weak convergence can be metricized in several ways. Among the many metrics giving the weak convergence, the Wasserstein or Monge-Kantorovich distance $W_1$ is especially useful for us.
The $W_1$ distance between $f,g\in P(K)$ is defined as
\begin{equation}\label{W1}
 W_1(f,g) = \inf_\pi \int_{K\times K} |\varpi-\varpi_*| \,\pi(d\varpi,d\varpi_*)
= \sup_\phi \, \int_K \phi(\varpi) \, (f -g)(d\varpi)
\end{equation}
where the $\inf$ is taken over all $\pi\in P(K\times K)$ with marginals $f$ and $g$, and 
the $\sup$ is taken over all the functions $\phi:K\to \R$ that are 1-Lipschitz.
The last equality is indeed the Kantorovich-Rubinstein Theorem. 
We refer to the book \cite{V} for more details concerning Monge-Kantorovich distances.

\section{Description of the model}

\subsection{Opinion, hierarchy and stubbornness parameters}

We consider a large hierarchical population of individuals and suppose that each one of them is characterized by three parameters: their opinion $w$, their hierarchy level $h$, and their stubbornness $q$. The opinion $w$ of an individual quantifies their attitude about some given topic currently discussed in the population. We model it by a real number in $[-1,1]$ where $w=\pm 1$ corresponds to a radical opinion and $w=0$ to a neutral one. 
The importance in the society of an individual is indicated by their hierarchy level $h$, a real number in $[0,1]$ where $h=1$ is the highest hierarchy level and $h=0$ the lowest.
Eventually the stubbornness parameter $q$ measures how stubborn is the individual when changing opinion.
We represent it with a real number $q$ in $[0,1]$, the probability the individual will accept being influenced by others in an interaction.
Notice in particular the individuals with $q=0$ are stubborn in the sense they will never change their opinion.
The impact of stubborn individuals on the formation of consensus and its value is important as shown in \cite{ppss}.
Indeed the analysis in \cite{ppss} shows that the stubborn individuals plays a critical role in the opinion formation process as they both accelerate the consensus formation in the non-stubborn population and determine the asymptotic consensus opinion.
In the non-stubborn population, the precise value of $q$ only affects the velocity at which individuals change opinion but not the qualitative overall dynamic. Since the main focus of this paper is on the impact of the hierarchy parameter $h$ on the opinion formation process, we will assume for simplicity that $q$ can only take the values $0$ and $1$. Thus the stubborn individuals have $q=0$, and the non-stubborn ones have $q=1$.

\subsection{Interaction rule}

The most crucial point in the modelling of an opinion formation process consists in specifying the interactions between individuals since they will result in changes of the opinion of the interacting individuals and will ultimately be the reason for any macroscopic property of the distribution of opinions.

Here we follow the majority of the papers on this topic assuming that interactions occur at a constant rate (assumed w.l.o.g to be 1) between pair of randomly chosen individuals.

Assume e.g. that agents $i$ and $j$ with parameters $(w,h,q)$ and $(w_*,h_*,q_*)$ are to interact.
We denote their new post-interaction parameters by $(w',h',q')$ and $(w'_*,h'_*,q_*')$ respectively.
We assume for simplicity that the hierarchy level and the stubbornness do not change:
$$h'=h, \quad h'_*=h_*, \qquad \text{and}\qquad
q'=q, \quad q'_*=q_*. $$
If agent $i$ is stubborn, i.e. $q=0$, then he will not change his mind:
$$ w'= w \qquad \text{if $q=0$.} $$
If he is not stubborn, i.e. $q=1$, agent i will change his opinion under the influence of $j$
if either $j$ have a higher or equal hierarchy than $i$,
 i.e. $h_*\ge h$, or, when $h_*< h$, with a given probability $p\in [0,1]$.
In both cases, $i$ is convinced by $j$ and as a result slightly moves his opinion $w$ toward $j$'s opinion $w_*$:
\begin{equation}\label{InterRules}
 w'=
\begin{cases}
w+\gamma (w_*-w) \qquad & \text{if $h_{*}\ge h$, or, if $h_{*} < h$, with probability $p$, } \\
w  & \text{otherwise.}
\end{cases}
\end{equation}
Here $\gamma>0$ is a given parameter modelling the strength of the interaction.

Interaction rule \eqref{InterRules} corresponds to an attractive interaction which models the tendency to compromise. 
This is a well-studied \cite{Abelson} mechanism considered by different sociological theories (persuasion \cite{Akers, Vinokur}, imitation \cite{Akers}, social pressure \cite{Asch, Sherif}). 

The parameter $p$ models the possibility of not respecting the hierarchy i.e. the probability that a lower ranked individual convinces a higher ranked one. Notice that the extreme case $p=0$ means that the hierarchy is always strictly respected. We refer to this case as the ``pure hierarchy model''. We expect in that case a vertical transmission of information from the highest hierarchy level down to the lowest, eventually
 altered at each level by the stubborn individuals. We will confirm this intuition later.
On the other hand, when $p$ is positive, this perfect vertical propagation of opinion is disturbed.
It is then interesting to see to the impact of the disturbance on the formation of consensus.
Eventually when $p=1$ then $j$ always convinces $i$ whatever their hierarchy levels: hierarchy is irrelevant. This is the case studied in \cite{ppss}.
We thus focus on the case $p=0$ where hierarchy is strictly respected, and the case $p\in (0,1)$ where hierarchy is disrupted with probability $p$.

\section{Macroscopic kinetic model}

\subsection{Macroscopic integral equation}

To describe at the whole population level the consequences of the microscopic interaction rule presented in the previous section,
we follow the methodology used e.g. in \cite{Bell, PT, GT}. 
We thus introduce the  distribution $f_\gamma(t,d\varpi)$ of the pair $\varpi$ in the whole population at time $t$.
Notice $f_\gamma(t,d\varpi)$ belongs to $P(K)$, the set of probability measures on $K:=[-1,1]\times [0,1]$.
Then
\begin{equation}\label{DefMeanOpinion}
m_\gamma(t) = \int_K w\,f_\gamma(t,d\varpi)
\end{equation}
is the expected mean opinion in the whole population, and
\begin{equation}\label{Defvarianza}
Var_\gamma(t)  = \int_K w^2\,f_\gamma(t,d\varpi) - m_\gamma(t)^2
\end{equation}
is the variance of $f(t,d\varpi)$.
In general for any continuous function $\phi\in C(K)$, the integral $\int_K\phi(\varpi)\,f_\gamma(t,d\varpi)$ is the expected  mean value of $\phi$ at time $t$.
The time evolution of $f_\gamma(t,d\varpi)$ is then characterized through the time evolution of $\int_K\phi(\varpi)\,f_\gamma(t,d\varpi)$ for any $\phi\in C(K)$.
Following \cite{GT} $f_\gamma(t,d\varpi)$ satisfies an integral equation given in weak form by
\begin{equation}\label{BoltzmannEqu}
\begin{split}
 & \dfrac{d}{dt} \int_K \phi(\varpi)\,f_\gamma(t,d\varpi)
  = \int_{K\times K} \mathbb{E}[\phi(\varpi^{\prime})-\phi(\varpi)]\, f_\gamma(t,d\varpi)f_\gamma(t,d\varpi_*)
\end{split}
\end{equation}
for any $\phi\in C(K)$, where $\mathbb{E}[\phi(\varpi^{\prime})-\phi(\varpi)]$ is the expected value of $\phi(\varpi^{\prime})-\phi(\varpi)$.

A classical argument based on Banach fixed-point theorem shows that this equation is well-posed when we endow $P(K)$
with the total variation norm \eqref{TVnorm}.

\begin{thm}\label{existe}
For any initial condition $f_0\in P(K)$ there exists a unique $f_\gamma \in C([0, +\infty ), P(K))\cap C^1((0, +\infty ), P(K))$ satisfying (\ref{BoltzmannEqu}) with initial condition $f_{\gamma | t=0}=f_0$.
\end{thm}
\noindent The proof of this result is standard, see \cite{CIP} for the details. 

\medskip

The above equation \eqref{BoltzmannEqu} does not distinguish between stubborn and non-stubborn individuals although only non-stubborn change their opinion in time.
To obtain an equation for the evolution of the distribution of $\varpi$ in the non-stubborn population,
denote $\alpha\in [0,1]$ the proportion of stubborn agents, and $f_\gamma^S(t,d\varpi)$, $f_\gamma^{NS}(t,d\varpi)$ the distribution of opinion in the stubborn and non-stubborn population. 
Notice $\alpha$ is constant in time since the dynamic does not affect the stubbornness. 
As the stubborn individuals do not change opinion, the distribution of opinion in the stubborn population is indeed constant in time so that $f_\gamma^S(t,d\varpi)=f_0^S(d\varpi)$.
We then have
$$ f_\gamma(t,d\varpi)=\alpha f_0^S(d\varpi) + (1-\alpha)f_\gamma^{NS}(t,d\varpi). $$
Equation \eqref{BoltzmannEqu} can then be written as
\begin{eqnarray*}
 && \alpha\dfrac{d}{dt} \int_{K} \phi(\varpi)\,f^{S}_0(d\varpi) + (1-\alpha)\dfrac{d}{dt} \int_{K} \phi(\varpi)\,f_\gamma^{NS}(t,d\varpi) \\
& & = \alpha \int_{K^2} \mathbb{E}[\phi(\varpi')-\phi(\varpi)] \, f^S_0(d\varpi)f_\gamma(t,d\varpi_{*}) \\
&& \quad + (1-\alpha) \int_{K^2} \mathbb{E}[\phi(\varpi')-\phi(\varpi)] \, f_\gamma^{NS}(t,d\varpi)f_\gamma(t,d\varpi_{*}).
\end{eqnarray*}
The first derivative in the left hand side is clearly zero. 
Moreover since $w'=w$ when the agent is stubborn, the first integral in the right hand side is zero.
Thus
\begin{equation*}
\begin{split}
 & \dfrac{d}{dt} \int_{K} \phi(\varpi)\,f_\gamma^{NS}(t,d\varpi)
 =  \int_{K^2} \mathbb{E}[\phi(\varpi')-\phi(\varpi)] \, f_\gamma^{NS}(t,d\varpi)f_\gamma(t,d\varpi_{*}) \\
 &  =  \int_{K^2} [\phi(w+\gamma(w_*-w),h)-\phi(w,h)](p1_{h_*< h} + 1_{h_*\ge h}) \, f_\gamma^{NS}(t,d\varpi)f_\gamma(t,d\varpi_*)
\end{split}
\end{equation*}
i.e.
\begin{equation}\label{BoltzmannEqu2}
\begin{split}
 & \dfrac{d}{dt} \int_{K} \phi(\varpi)\,f_\gamma^{NS}(t,d\varpi) \\
 &  =  \int_{K^2} [\phi(w+\gamma(w_*-w),h)-\phi(w,h)](p + (1-p)1_{h_*\ge h}) \, f_\gamma^{NS}(t,d\varpi) f_\gamma(t,d\varpi_*).
\end{split}
\end{equation}

Studying the asymptotic behaviour of the solution $f_\gamma^{NS}(t,d\varpi)$ as $t\to +\infty$ of this integral equation is non trivial.
Notice for instance that taking $\phi(\varpi)=w$ does not yield a closed equation for the mean opinion in the non-stubborn population 
(see e.g. the system of equations \eqref{SystemMeanOpDiscrete} solved by the mean opinion in each non-stubborn subgroup of a given hierarchy). 
Instead we will rely on a procedure called grazing limit to deduce from \eqref{BoltzmannEqu2} a local equation amenable to analysis.

\subsection{Grazing limit: Heuristic}

The grazing  limit procedure is  well-known in the mathematical literature on the Boltzmann equation and has been adapted to opinion formation model in \cite{GT}. 
In our setting it consists simply in taking $\gamma\ll 1$ and approximating
\[ \phi(w+\gamma(w_*-w),h)-\phi(w,h) \simeq  \gamma  \phi_w(\varpi)(w_*-w). \]
Then \eqref{BoltzmannEqu2} becomes
\begin{align}
 &\frac{1}{\gamma} \dfrac{d}{dt} \int_{K} \phi(\varpi) \,f_\gamma^{NS}(t,d\varpi) \notag \\
& \approx p \int_{K^2} \phi_w(\varpi)(w_*-w)\,f_\gamma^{NS}(t,d\varpi)f_\gamma(t,d\varpi_{*} ) \notag \\
& \hspace{0.5cm} + (1-p)\int_{K^{2}}
\phi_w(\varpi)(w_*-w)1_{h_*\ge h}\,f_\gamma^{NS}(t,d\varpi)f_\gamma(t,d\varpi_*). \label{Equ100}
\end{align}
The first integral in the right hand side is
$$  \int_{K^2} \phi_w(\varpi)(m_\gamma(t)-w)\,f_\gamma^{NS}(t,d\varpi)   $$
where $m_\gamma(t)$ is the mean opinion defined in \eqref{DefMeanOpinion}.
Notice that
$$ \int_{K}1_{h_*\ge h}\,f_\gamma(t,d\varpi_{*}) = f_0([h,1]) $$
is the proportion of agents with a hierarchy greater than or equal to $h$
(which remains constant time since the dynamic we consider here do not affect the hierarchy level).
Denote
\begin{eqnarray}\label{DefOp}
  b_\gamma(t,h) := \int_{K} w_* 1_{h_*\ge h} \frac{f_\gamma(t,d\varpi_*)}{f_{0}([h,1])}
\end{eqnarray}
the average opinion among agents with a hierarchy greater than or equal to $h$
(notice that $\frac{f_\gamma(t,d\varpi_*) }{f_{0}([h,1])}$ is a probability measure on $[-1,1]\times [h,1]$).
The second integral in the right hand side of \eqref{Equ100} is then
$$ \int_{K} \phi_w(\varpi) (b_\gamma(t,h)-w) f_{0}([h,1]) \,f_\gamma^{NS}(t,d\varpi).$$
We thus obtain from \eqref{Equ100} the  equation
\begin{equation}\label{LimitEqu}
\begin{split}
& \frac{1}{\gamma} \dfrac{d}{dt}\int_{K}\phi(\varpi)\,f_\gamma^{NS}(t,d\varpi) \\
& =  p\int_K \phi_w(\varpi) \left( m_\gamma(t) - w\right) \,f_\gamma^{NS}(t,d\varpi) \\
& \quad +  (1-p)\int_K \phi_w(\varpi)\left( b_\gamma(t,h)- w \right) f_{0}([h,1])\,f_\gamma^{NS}(t,d\varpi).
\end{split}
\end{equation}
This equation is the weak formulation of the transport equation
\begin{equation}\label{LimitEq}
 \frac{1}{\gamma} \frac{\p}{\p t} f_\gamma^{NS} +
\frac{\p}{\p w} \Big( v[f_\gamma^{NS}(t)](w) f_\gamma^{NS} \Big) = 0, 
\end{equation}
where 
$$ v[f_\gamma^{NS}(t)](w) = p(m_\gamma(t) - w)+(1-p)(b_\gamma(t,h)- w)f_{0}([h,1]). $$   
The grazing limit procedure thus allowed to replace the integral equation \eqref{BoltzmannEqu2} by the local equation \eqref{LimitEq} 
which is  expected to  approximate  \eqref{BoltzmannEqu} in the limit $\gamma\simeq 0$.

\subsection{Existence of the grazing limit}

These  heuristic considerations can be justified up to rescaling time considering the new time-scale $\tau=\gamma t$.
Indeed since $\gamma\simeq 0$, changes in opinion on the time scale $t$ are infinitesimal.
We then have to wait a long time $\tau$ to see a macroscopic change.

The following results aim at justifying the above  informal derivation.  
From now on we endow $P(K)$ with the weak convergence metricized e.g. by the $W_1$ distance \eqref{W1}.

As a first step we can prove that $f_\gamma$ converges as $\gamma\to 0$ along a subsequence on the time scale $\tau$. 

\begin{thm}\label{GrazingConvergence}
Given an initial condition $f_{0} \in P(K)$, consider the solution $f_\gamma$ of
equation (\ref{BoltzmannEqu}) as given by Theorem \ref{existe}.
With a slight abuse of notation we denote  $f_{\gamma}(\tau,d\varpi) := f_\gamma(t,d\varpi)$ where $\tau = \gamma t$.
Then there exists $f^{NS}\in C([0, +\infty), P(K))$ such that, up to subsequences, $f^{NS}_{\gamma}\rightarrow f^{NS}$ as $\gamma\rightarrow 0$
in $C([0, T], P(K))$ for any $T > 0$.
\end{thm}

\begin{rem}\label{Remark}
Since $f_\gamma(\tau,d\varpi) = \alpha f_0^S(d\varpi) + (1-\alpha) f_\gamma^{NS}(\tau,d\varpi)$ and $f_0^S$ is constant in time, we obtain that
$f_\gamma\to f:=\alpha f_0^S + (1-\alpha) f^{NS}$.
\end{rem}

\begin{proof}
We can see from \eqref{BoltzmannEqu2} that the time-rescaled measure $f_{\gamma}(\tau)$ solves
\begin{equation*}
\begin{split}
 & \gamma\dfrac{d}{d\tau} \int_{K} \phi(\varpi)\,f^{NS}_\gamma(\tau,d\varpi) \\
 &  =  \int_{K^2} [\phi(w+\gamma(w_*-w),h)-\phi(w,h)](p + (1-p)1_{h_*\ge h}) \, f^{NS}_\gamma(\tau,d\varpi)f_\gamma(\tau,d\varpi_*).
\end{split}
\end{equation*}
We use a Taylor expansion of $\phi$ with respect to the $w$ variable:
\begin{align*}
 \phi(w+\gamma(w_*-w),h)-\phi(w,h)
& = \gamma\phi_{w}(\varpi)(w_*-w) + \dfrac{\gamma^2}{2}\phi_{ww}(\xi,h)(w_* -w)^{2}
\end{align*}
where $\xi$ lies between $w$ and $w+\gamma(w_*-w)$.
Then the same considerations that led to \eqref{LimitEqu} gives
\begin{eqnarray*}
&& \dfrac{d}{d\tau}\int_{K}\phi(\varpi)f_{\gamma}^{NS}(\tau ,d\varpi)
 =  p\int_{K}\phi_{w} \left( m_\gamma(\tau) - w\right) f_{\gamma}^{NS}(\tau ,d\varpi) + R(\tau ,\gamma) \\
&& +  (1-p)\int_{K}\phi_{w}\left( b_\gamma(\tau,h)- w \right) f_{0}\left([h,1]\right)f^{NS}_{\gamma}(\tau ,d\varpi),
\end{eqnarray*}
where
\[R(\tau ,\gamma)= \dfrac{\gamma}{2}\int_{K^{2}} (w_*-w)^2 \phi_{ww}(\xi,h)
(p + (1-p)1_{h_*\ge h}) \, f^{NS}_\gamma(\tau,d\varpi)f_\gamma(\tau,d\varpi_*).  \]
Integrating in time between $\tau$ and $\tau'$ we obtain
\begin{equation}\label{ArzelaAscoli}
\begin{split}
& \int_{K}\phi(\varpi) (f^{NS}_{\gamma}(\tau' ,d\varpi)-f^{NS}_{\gamma}(\tau ,d\varpi))
 =  p \int_\tau^{\tau'}\int_K \phi_w(\varpi) \left( m_\gamma(s) - w\right) f_{\gamma}^{NS}(s,\varpi)ds \\
& +  (1-p)\int_\tau^{\tau'} \int_K\phi_{w}(\varpi)\left( b_\gamma(s,h)- w \right) f_0\left([h,1]\right)f^{NS}_{\gamma}(s,d\varpi)ds\\
 &+ \int_\tau^{\tau'} R(s,\gamma)\,ds.\\
\end{split}
\end{equation}
Using that $|p + (1-p)1_{h_*\ge h}|\le 1$ and $(w_*-w)^2\le 4$ we have $|R(s,\gamma)|\le 2\gamma \|\phi_{ww}\|_\infty$.
Moreover $m_\gamma(s),b_\gamma(s,h)\in [-1,1]$ so that $|m_\gamma(s) -w|\le 2$ and  $|b_\gamma(s,h)-w|\le 2$.
Thus
\begin{eqnarray*}
\Big| \int_{K}\phi(\varpi) (f_{\gamma}^{NS}(\tau' ,d\varpi)-f_{\gamma}^{NS}(\tau ,d\varpi))\Big| \le C\|\phi\|_{C^2} |\tau'-\tau|
\end{eqnarray*}
where $C= 4+2\gamma$ and $\|\phi\|_{C^2}=\|\phi\|_\infty +  \|\phi_w\|_\infty + \|\phi_{ww}\|_\infty$.
We thus obtain
\begin{equation}\label{equ11}
\sup_{\phi \in C^3(K), \, \|\phi\|_{C^2}\leq 1}
\left|\int_{K}\phi(\varpi)\left( f_{\gamma}^{NS}(\tau^{\prime} ,d\varpi)-f_{\gamma}^{NS}(\tau ,d\varpi)\right) \right|\leq  C (\tau^{\prime}-\tau).
\end{equation}
The duality with $C^2(K)$ defines the following norm $\|.\|$ on $P(K)$:
\begin{equation*}
\|f\| := \sup_{\phi \in C^2(K), \, \|\phi\|_{C^2}\leq 1} \int_K\phi(\varpi ) \,d\mu.
\end{equation*}
According to Lemma 5.3 and Corollary 5.5 in \cite{GTW} this norm induces the weak topology on $P(K)$.

Notice that \eqref{equ11} implies that the family $\{f_\gamma^{NS}\}_\gamma$ is uniformly equicontinuous in $C([0,+\infty),(P(K),\|.\|))$.
Since $P(K)$ is compact for the weak convergence, we can then apply Arzela-Ascoli Theorem in $C([0,T],P(K))$, $T>0$, to obtain the existence of $f^{NS}\in C([0,+\infty),P(K))$ such that
up to a subsequence as $\gamma\to 0$, $f^{NS}_\gamma\to f^{NS}$ in $C([0,T],P(K))$, $T>0$.
\end{proof}

The next natural step  consists in  passing to the limit $\gamma\to 0$ in
\begin{equation}\label{CasiLimitEq}
\begin{split}
& \int_K\phi(\varpi)\, f^{NS}_{\gamma}(\tau,d\varpi) = \int_K\phi(\varpi) \, f^{NS}_0(d\varpi)\\
& \quad +  p \int_0^\tau\int_K \phi_w(\varpi) \left( m_\gamma(s) - w\right) f_{\gamma}^{NS}(s,\varpi)ds \\
& \quad +  (1-p)\int_0^\tau \int_K\phi_w(\varpi)\left( b_\gamma(s,h)- w \right) f_0\left([h,1]\right)f^{NS}_{\gamma}(s,d\varpi)ds\\
& \quad  + \int_0^\tau R(s,\gamma)\,ds   \\
\end{split}
\end{equation}
for a given $\phi\in C^2(K)$ to deduce that $f^{NS}$ solves the limit equation \eqref{LimitEqu}.
Clearly $\int_0^\tau R(s,\gamma)\,ds \to 0$ since $|R(s,\gamma)|\le 2\gamma \|\phi_{ww}\|_\infty$.
Moreover the convergence $f_\gamma^{NS}\to f^{NS}$ we just proved allows to pass to the limit in the left hand side. 
Also 
\begin{eqnarray*}
 m_\gamma(s)  =  \int_{-1}^1 w\,f_\gamma(s,d\varpi)
= \alpha \int_{-1}^1 w\,f^S_0(d\varpi) + (1-\alpha) \int_{-1}^1 w\,f^{NS}_\gamma(s,d\varpi)
\end{eqnarray*}
converges uniformly for $s\in [0,T]$, $T>0$, to
 $$ \alpha \int_{-1}^1 w\,f^S_0(d\varpi) + (1-\alpha) \int_{-1}^1 w\,f^{NS}(s,d\varpi) =: m(s). $$
We can thus pass to the limit in the second term in the right hand side.

However the third term in the right hand side of \eqref{CasiLimitEq} is not trivial to handle because  the functions
$$ (w,h)\to (b_\gamma(s,h)- w) f_0([h,1]) = \int_K w_* 1_{h_*\ge h} \, f_\gamma(s,d\varpi_*) - wf_0([h,1]) $$
and $(w_*,h_*)\to  w_* 1_{h_*\ge h}$  
are in general not continuous.

We were able to circumvent this difficulty when there are only a finite number of hierarchy levels in the population,

\subsection{Limit equation when there is a finite number of hierarchy level.}\label{SectionDiscrete}

Let us consider the case where there is a finite number $N$ of hierarchy levels $0\le h_1<h_2<...<h_N\le 1$.
We need to introduce some notations.
We denote $f_{0}(h_i)$ the proportion of individuals with hierarchy level $h_i$, $i=1,...,N$,
and $f_{0}^S(h_i)$ and $f_{0}^{NS}(h_i)$ the proportion of individuals with hierarchy level $h_i$ within the stubborn and non-stubborn
population. Thus $f_{0}(h_i) = \alpha f_{0}^S(h_i) + (1-\alpha)f_{0}^{NS}(h_i)$.
We can then write the distribution of hierarchy level as
\begin{eqnarray}\label{f0dh}
f_0(dh)= \sum_{i=1}^{N} f_0(h_i)\delta_{h_i} = \alpha f^S_0(dh) + (1-\alpha) f_0^{NS}(dh)
\end{eqnarray}
where $ f^S_0(dh) = \sum_{i=1}^{N} f_0^S(h_i)\delta_{h_i}$ and $ f_0^{NS}(dh)= \sum_{i=1}^{N} f_0^{NS}(h_i)\delta_{h_i}$.

In the same way we write  the measure $f_\gamma^{NS}(\tau,d\varpi)$ as
\begin{eqnarray}\label{fgama}
f_\gamma^{NS}(\tau,d\varpi) = \sum_{i=1}^N f_0^{NS}(h_i)\left(f_{\gamma}^{NS,i}(\tau,dw) \otimes  \delta_{h_{i}} \right)
\end{eqnarray}
where $f_{\gamma}^{NS,i}(\tau,dw)$ is the distribution of opinion in the non-stubborn $h_i$-population.

Equation \eqref{CasiLimitEq} then becomes a system of N coupled equations for $f_{\gamma}^{NS,i}(\tau,dw)$, $i=1,...,N$.
Indeed \eqref{CasiLimitEq} reads
\begin{equation*}
\begin{split}
& \sum_{i=1}^N f_0^{NS}(h_i) \int_{-1}^1 \phi(w,h_i)\, f^{NS,i}_{\gamma}(\tau,dw)
= \sum_{i=1}^N f_0^{NS}(h_i) \int_{-1}^1 \phi(w,h_i) \, f^{NS,i}_0(dw) \\
&
+  p \sum_{i=1}^N f_0^{NS}(h_i) \int_0^\tau\int_{-1}^1  \phi_w(w,h_i) \left( m_\gamma(s) - w\right) f_{\gamma}^{NS,i}(s,w)ds \\
&  +  (1-p)\sum_{i=1}^N f_0^{NS}(h_i) \int_0^\tau \int_{-1}^1  \phi_w(w,h_i)\left( b_\gamma(s,h_i)- w \right) f_0\left([h_i,1]\right)f^{NS,i}_{\gamma}(s,dw)ds\\& + \int_0^\tau R(s,\gamma)\,ds,
\end{split}
\end{equation*}
i.e. for any $i=1,...,N$, and for any $\phi\in C^2([-1,1])$,
\begin{equation}\label{CasiLimit2}
\begin{split}
& \int_{-1}^1 \phi(w)\, f^{NS,i}_{\gamma}(\tau,dw)
= \int_{-1}^1 \phi(w) \, f^{NS,i}_0(dw) \\
& \quad +  p \int_0^\tau\int_{-1}^1  \phi'(w) \left( m_\gamma(s) - w\right) f_{\gamma}^{NS,i}(s,dw)ds \\
& \quad  +  (1-p) \int_0^\tau \int_{-1}^1  \phi'(w)\left( b_\gamma(s,h_i)- w \right) f_0\left([h_i,1]\right)f^{NS,i}_{\gamma}(s,dw)ds\\
& \quad + O(\gamma)\tau \\
\end{split}
\end{equation}
where we used that $|R(s,\gamma)|\le 2\gamma \|\phi_{ww}\|_\infty$.

We can now pass to the limit $\gamma\to 0$. Recall from Remark \ref{Remark} that we write the limit $f$ of $f_\gamma$ as
\begin{equation}\label{Notation1}
 f(\tau,d\varpi)=\alpha f^S_0(d\varpi) +(1-\alpha)f^{NS}(\tau,d\varpi), 
\end{equation}
with 
\begin{equation}\label{Notation2}
 f^{NS}(\tau,d\varpi) = \sum_{i=1}^N f_0^{NS}(h_i) (f^{NS,i}(\tau,dw)\otimes \delta_{h_i})
\end{equation}
in the same way as \eqref{fgama}. 
Denote also
\begin{equation}\label{Equ102}
 m(\tau)=\int_{-1}^1 w\,f(\tau,d\varpi), \qquad
b(\tau,h) = \int_{K} w_* 1_{h_*\ge h} \frac{f(\tau,d\varpi_*)}{f_{0}([h,1])}.
\end{equation}
We then have

\begin{thm} \label{GrazingLimitDiscrete}
The measures $f^{NS,i}(\tau,dw)$, $i=1,..,N$, solves the system
\begin{equation}\label{eqdiscreta}
\begin{split}
& \int_{-1}^1 \phi(w)\, f^{NS,i}(\tau,dw)
= \int_{-1}^1 \phi(w) \, f^{NS,i}_0(dw) \\
&   +  p \int_0^\tau\int_{-1}^1  \phi'(w) \left( m(s) - w\right) f^{NS,i}(s,dw)ds \\
&  +  (1-p) \int_0^\tau \int_{-1}^1  \phi'(w)\left( b(s,h_i)- w \right) f_0\left([h_i,1]\right)f^{NS,i}(s,dw)ds
\end{split}
\end{equation}
for any $\phi\in C^1([-1,1])$.
\end{thm}

\begin{proof}
The convergence $f_\gamma^{NS}\to f^{NS}$ proved in Theorem \ref{GrazingConvergence} is equivalent
to $f_\gamma^{NS,i}\to f^{NS,i}$, $i=1,...,N$, i.e. for any $\phi\in C([-1,1])$, 
$$ \int_{-1}^1 \phi(w)\, f^{NS,i}_{\gamma}(\tau,dw) \to \int_{-1}^1 \phi(w)\, f^{NS,i}(\tau,dw) $$
uniformly for $\tau\in [0,T]$, $T>0$.

We want to pass to the limit $\gamma\to 0$ in \eqref{CasiLimit2}.
As explained right after \eqref{CasiLimitEq}, we only need to concentrate on the last term in \eqref{CasiLimit2}, namely
\begin{eqnarray*}
 && \int_0^\tau \int_{-1}^1  \phi'(w)  b_\gamma(s,h_i) f_0([h_i,1])\, f^{NS,i}_{\gamma}(s,dw)ds  \\ 
 &&  - f_0([h_i,1])\int_0^\tau \int_{-1}^1  \phi'(w) w \, f^{NS,i}_{\gamma}(s,dw)ds =: A-B.
\end{eqnarray*}
We can clearly pass to the limit in the second term B. Concerning the first term A, notice first that 
\begin{equation*}\label{terminodiscreto}
\begin{split}
 f_0\left([h_i,1]\right) b_\gamma(s,h_i) 
   & =  \alpha \sum_{j\ge i} f_0^S(h_j)\int_{-1}^1 w_*  f_0^{S,j}(dw_*) \\ 
& \quad + (1-\alpha) \sum_{j\ge i} f_0^{NS}(h_j) \int_{-1}^1  w_*  f_\gamma^{NS,j}(s,dw_*).
\end{split}
\end{equation*}
Letting
$$ m_0^{S,j} = \int_{-1}^1 w_*  f_0^{S,j}(dw_*), \qquad m_\gamma^{NS,j}(s) = \int_{-1}^1 w_*  f_\gamma^{NS,j}(s,dw_*) $$
be the mean opinion in the stubborn and non-stubborn $h_j$-population, we obtain
\begin{eqnarray*}
A & =& \alpha \sum_{j\ge i} f_0^S(h_j) m_0^{S,j} \int_0^\tau \int_{-1}^1  \phi'(w)\, f^{NS,i}_{\gamma}(s,dw)ds   \\
&&+ (1-\alpha) \sum_{j\ge i} f_0^S(h_j)   m_\gamma^{NS,j}(s) \int_0^\tau\int_{-1}^1  \phi'(w)  \, f^{NS,i}_{\gamma}(s,dw)ds.
\end{eqnarray*}
Since
$$\int_{-1}^1  \phi'(w)\, f^{NS,i}_{\gamma}(s,dw)\to \int_{-1}^1  \phi'(w)\, f^{NS,i}(s,dw)$$ and
$$m_\gamma^{NS,j}\to \int_{-1}^1 w_*  f^{NS,j}(s,dw_*) =: m^{NS,j}$$ uniformly in $[0,T]$, $T>0$, we can pass to the limit
in $A$.
\end{proof}

\subsection{A regularized equation}

Introducing the function $J(h)=1_{h\le 0} + p1_{h>0}$ we can rewrite the opinion updating rule \eqref{InterRules} as $w'=w + \gamma(w_*-w)$ with probability $J(h-h_*)$. 
Equation \eqref{BoltzmannEqu2} becomes 
\begin{equation*} 
\begin{split}
 & \dfrac{d}{dt} \int_{K} \phi(\varpi)\,f_\gamma^{NS}(t,d\varpi) \\
 &  =  \int_{K^2} [\phi(w+\gamma(w_*-w),h)-\phi(w,h)]J(h-h_*) \, f_\gamma^{NS}(t,d\varpi) f_\gamma(t,d\varpi_*).
\end{split}
\end{equation*}
which yields in the limit $\gamma\approx 0$ the equation 
\begin{equation*} 
\begin{split}
 & \dfrac{d}{d\tau} \int_K \phi(\varpi)\,f^{NS}(\tau,d\varpi) \\
 &  =  \int_K  \phi_w(\varpi) \Big\{ \int_K (w_*-w)J(h-h_*) f_\gamma(t,d\varpi_*) \Big\} \, f_\gamma^{NS}(t,d\varpi).
\end{split}
\end{equation*}
Notice the main difficulty we faced before in order to justfy the limit $\gamma\to 0$ and prove that the limit of the $f_\gamma$ satisfies 
this limit equation was the lack of regularity of $J$. 

We can circumvent this difficulty considering a smooth approximation $J_\eps$ such that $0\le J_\eps\le 1$ and $J_\eps(h)=1$ if $h\le 0$, 
$J_\eps(h)=p$ if $h\ge \eps$ for a small $\eps>0$. 
The limit equation when $\gamma\to 0$ is then 
\begin{equation*} 
\begin{split}
  \dfrac{d}{d\tau} \int_K \phi(\varpi)\,f^{NS}(\tau,d\varpi) 
   =  \int_K  \phi_w(\varpi) v_\eps[f(\tau)](\varpi)  f^{NS}(\tau,d\varpi)
\end{split}
\end{equation*}
where 
$$ v_\eps[f(\tau)](\varpi) =  \int_K (w_*-w)J_\eps(h-h_*) f(\tau,d\varpi_*), $$ 
which is the weak formulation of the transport equation 
$$ \frac{\p}{\p t} f^{NS} + \frac{\p}{\p w} \Big( v_\eps[f(t)](w) f^{NS} \Big) = 0. $$ 
Since $J_\eps$ is bounded and globally Lipschitz, this equation has a unique solution (see e.g. \cite{AS}). 
The same kind of argument as before easily shows this unique solution is the limit of the $f_\gamma^{NS}$ when $\gamma\to 0$.
Notice eventually that  when there are finitely many hierarchy level $h_1,..,h_N$ then the original (with $J$) and the regularized dynamic (with $J_\eps$) coincide 
if $\eps\ll 1$ (namely if $\eps<\min_{i\neq j} |h_i-h_j|$).

\section{Long-time behaviour}

According to Theorem \ref{GrazingLimitDiscrete}, when the distribution $f_0(dh)$ of hierarchy level
is finite discrete, any solution $f^{NS}_\gamma$ of \eqref{BoltzmannEqu2} converges as $\gamma\to 0$
(up to a subsequence and time rescaling) to a solution of equation  \eqref{LimitEqu}, namely
\begin{equation}\label{LimitEqu5}
\begin{split}
& \dfrac{d}{dt}\int_{K}\phi(\varpi)\,f^{NS}(t,d\varpi)
 =  p\int_K \phi_w(\varpi) \left( m(t) - w\right) \,f^{NS}(t,d\varpi) \\
& \qquad +  (1-p)\int_K \phi_w(\varpi)\left( b(t,h)- w \right) f_{0}([h,1])\,f^{NS}(t,d\varpi). 
\end{split}
\end{equation}
From now on we denote the time by $t$ instead of $\tau$ for ease of notation.

In this section we investigate the long time behaviour of $f^{NS}$.

Notice first that when $p=1$, the hierarchy do not impact on the dynamic.
This is the case studied in \cite{ppss} where it was proved that $f^{NS}(t)$ converge as $t\to +\infty$
to a Dirac mass which is located at  the initial mean opinion of (i) the whole population, $m(0)$, if $\alpha=0$, 
and of (ii) the stubborn population, $m^S_0$, if $\alpha>0$. Thus as $t\to +\infty$, 
\begin{equation}\label{Casep1}
 f^{NS}(t)\to \delta_{m_\infty^{NS}},\qquad \text{where} \qquad
m_\infty^{NS} =
\begin{cases}
m(0) \quad & \text{if $\alpha=0$} \\
 m_0^S \quad & \text{if $\alpha>0$}.
\end{cases}
\end{equation}

\medskip

We assume from now that $p\in [0,1)$. 
We also assume that the hierarchy level can only take the values $0\le h_1<h_2<...<h_N\le 1$ and w.l.o.g that 
there are top-ranked individual in the population i.e. $f_0(h_N)>0$. 

Recall that $f^{NS,i}(t,dw)$ is the distribution of opinion in the non-stubborn population with hierarchy level $h_i$ 
(see \eqref{Notation1} and \eqref{Notation2}). 
Denote 
\begin{equation}\label{Notation3}
 m^{NS,i}(t):=m^{NS,h_i}(t)=\int_{K} w\,f^{NS,i}(t,dw)\qquad i=1,...,N 
\end{equation}
the mean opinion in the non-stubborn $h_i$ population, and 
\begin{equation}\label{Notation4}
 m^{S,i}_0:=m^{S,h_i}_0=\int_{K} w\,f^{S,i}_0(dw)\qquad i=1,...,N 
\end{equation}
the mean opinion in the stubborn $h_i$ population.

\medskip

We will first verify that the tendency-to-compromise modelled by the interaction rule \eqref{InterRules}
implies that the opinion dynamic is contractive: the distribution of opinion in a non-stubborn subpopulation of a given hierarchy
level $h_i$ shrinks to a point, namely its mean value $m^{NS,i}(t)$.
As a consequence, we will only need to study the asymptotic behaviour of the $m^{NS,i}(t)$, $i=1,..,N$.

\subsection{Contractive dynamic}


Given some $f \in P ([-1, 1])$, we recall that the cumulative distribution function (cdf) $F : \mathbb{R}\rightarrow [0, 1]$ of $f$ is defined as $F(x) = f((-\infty, x])$. 
The generalized inverse of $F$ is  $X:(0, 1] \rightarrow [-1, 1]$ defined by
\[ X(\rho ) = \inf\left\lbrace x \in [-1, 1]\, s.t. \; F(x) \geq \rho \right\rbrace .\]
Notice that $X(1)$ and $X(0^+):=\lim_{r\to 0^+}X(r)$ are the right and left end points of the support of $f$. 

The generalized inverse enables us to  rewrite  equation \eqref{eqdiscreta} satisfied by $f^{NS,i}$ 
in terms of the generalized inverse of its cdf. Using the resulting equation we will easily show that 
the support of $f^{NS,i}(t)$ shrinks to a point exponentially fast. 
To do so we will need the following result (see  Theorem 3.1 in \cite{ANT} and Prop. 3.1 in \cite{ppss}): 

\begin{prop} \label{EquGenInv}
Let $v:(t,x)\in [0,+\infty)\times [-1,1]\to v(t,x)\in \R$ be continuous in $(t,x)$ and globally Lipschitz in $x$. 
Then $f \in C([0,+\infty),P([-1,1]))$ is a weak solution of
\begin{equation*}
  \p_tf_t + \p_x(v(t,x)f_t) = 0,
\end{equation*}
in the sense that for any $\phi\in C^1([-1,1])$ and any $t>0$,
\begin{equation}\label{Equ300}
  \int_{-1}^1 \phi(x)\,df_t(x) =  \int_{-1}^1 \phi(x)\,df_0(x)
+ \int_0^t \int_{-1}^1 \phi'(x)v(s,x)\,df_s(x)ds,
\end{equation}
if and only if for any $r\in (0,1]$, $X_t(r)$ is a solution of
\begin{equation*}
  \p_tX_t(r) = v(t,X_t(r)).
\end{equation*}
\end{prop}

We can then prove that 

\begin{prop}\label{Prop_Mean_Mixed} For any $i=1,..,N$ and $t\ge 0$, 
\begin{equation}\label{SupportShrinks}
|conv(supp\,f^{NS,i}(t))| \le |conv(supp\,f_0^{NS,i})|e^{-(p+(1-p)f_0([h_i,1]))t}
\end{equation}
where $|conv(supp\,f^{NS,i})(t)|$ denotes the length of the convex hull of the support of
$f^{NS,i}(t)$.
\end{prop}

\begin{proof}
Notice that equation \eqref{eqdiscreta} satisfied by $f_t^{NS,i}$ can be written as \eqref{Equ300} with 
$v(t,x) =  p ( m(t) - x) +  (1-p)( b(t,h_i)- x) f_0([h_i,1])$. 
Thus according to Prop. \ref{EquGenInv}, the generalized inverse $X_t$ of the cdf of $f_t^{NS,i}$ satisfies 
$$ \p_tX_t(r)
= p(m(t) - X_t(r))+ (1-p)\left( b(t,h_i)- X_t(r) \right) f_{0}\left([h_i,1]\right) $$
for $r\in (0,1]$. 
Then for $s\in (0,1]$,
\begin{eqnarray*}
\p_t\left( X_t(1)-X_t(s)\right)^{2}
= -2|X_t(1)-X_t(s)|^2 (p+(1-p)f_0([h_i,1]))
\end{eqnarray*}
so that
$$ |X_t(1)-X_t(s)|\le |X_0(1)-X_0(s)|e^{-(p+(1-p)f_0([h_i,1]))t}. $$
Sending $s\to 0^+$ gives the result.
\end{proof}

Notice that $p+(1-p)f_0([h_i,1])>0$ for any $i=1,..,N$ since $f_0([h_i,1])\ge f_0(h_N)>0$. 
It thus follows from this result that the support of $f^{NS,i}$ shrink to a point which means that individuals tend 
to share the same opinion. 
Moreover lower-ranked individuals coordinate faster than higher-ranked ones due to the term $f_0([h_i,1])$ which embodies 
the higher hierarchy pressure faced by lower-ranked individual. 
 
As a consequence, to study the long-time behaviour of the $f^{NS,i}$, it is enough to study the mean opinion $m^{NS,i}(t)$ defined in \eqref{Notation3}.
Indeed it follows from the definition \eqref{W1} of the $W_1$-distance that  
\begin{equation}\label{Eq500}
W_{1}\left( f^{NS,i}(t),\delta_{m^{NS,i}(t)}\right)\le e^{-(p+(1-p)f_0([h_i,1]))t}. 
\end{equation}

The following result says that the mean opinion $m^{NS,i}$ satisfy a linear system:

\begin{prop}
The mean opinions $m^{NS,1},...,m^{NS,N}$ satisfy 
\begin{equation}\label{SystemMeanOpDiscrete}
\begin{split}
\frac{d}{dt} m^{NS,i}(t)
& =  \alpha \sum_{j=1}^N m_0^{S,j} (p+(1-p)1_{\{j\ge i\}}) f_0^S(h_j) \\
& \quad + (1-\alpha) \sum_{j=1}^N  (p+(1-p)1_{\{j\ge i\}}) f_0^{NS}(h_j) m^{NS,j}(t) \\
& \quad  - \Big(p+(1-p)f_0([h_i,1])  \Big)m^{NS,i}(t).
\end{split}
\end{equation}
\end{prop}

\begin{proof}
Taking $\phi(w)=w$ in \eqref{eqdiscreta} gives
\begin{equation}\label{mtNS}
\begin{split}
 \frac{d}{dt} m_t^{NS,i}
&= \int_{K}\left\{ p(m(t) - w) + (1-p)( b(t,h_i)- w ) f_{0}\left([h_i,1]\right)\right\} \,f^{NS,i}(t,dw)\\
 &=p(m(t) - m^{NS,i}(t)) + (1-p)(b(t,h_i)-m^{NS,i}(t)) f_0([h_i,1]).\\
\end{split}
\end{equation}
Moreover recalling notations \eqref{Notation1}, \eqref{Notation2}, \eqref{Notation3}, \eqref{Notation4}, 
$$ m(t) = \alpha \sum_j f_0^S(h_j)m_0^{S,j}
+ (1-\alpha) \sum_j f_0^{NS}(h_j)m^{NS,j} $$ 
and
\begin{eqnarray*}\label{fbt}
&& f_0\left([h_i,1]\right) b(t,h_i) := \int_K w_* 1_{h_*\ge h_i} f(t,d\varpi_*) \\
 &  & = \alpha \sum_{j\ge i} f_0^S(h_j)m_0^{S,j}
+ (1-\alpha) \sum_{j\ge i} f_0^{NS}(h_j)m^{NS,j}.
\end{eqnarray*}
Replacing in \eqref{mtNS} gives the result.
\end{proof}

%
%


\subsection{Long-time behaviour}

In this section we study the asymptotic behaviour of the solution of system \eqref{SystemMeanOpDiscrete}. 
Notice first that system \eqref{SystemMeanOpDiscrete} is a linear system 
$$ X'(t)=AX(t)+B \qquad \text{ for } X(t)=(m^{NS,1}(t),...,m^{NS,N}(t))^T $$ 
with
$$B=(B_1,...,B_N)^T,\qquad
B_i=\alpha \sum_{j=1}^N m_0^{S,j} (p+(1-p)1_{\{j\ge i\}}) f_0^S(h_j),$$
and matrix $A=(A_{ij})$ given by
$$ A_{ij} =
\begin{cases}
 (1-\alpha)f_0^{NS}(h_i)-\Big(p+(1-p)f_{0}\left([h_i,1]\right)\Big), \qquad & i=j, \\
 (1-\alpha)  \Big[p+(1-p)1_{\{j\ge i\}} \Big] f_0^{NS}(h_j), \qquad &  i\neq j .
\end{cases}
$$

We now analyse separately the cases $p=0$ of ``pure hierarchy",  and the case $p\in (0,1)$ when the
hierarchical transition of
information is disrupted with probability $p$.

\subsubsection{Analysis of the ``pure hierarchy" case $p=0$.}

If $p=0$  then
$$ B_i = \alpha \sum_{j=i}^n m_0^{S,j} f_0^S(h_j),\qquad i=1,...,n, $$
and $A$ is triangular superior:
$$ A_{ij} =
\begin{cases}
 (1-\alpha)f_0^{NS}(h_i)-f_{0}\left([h_i,1]\right), \qquad & i=j, \\
 (1-\alpha)  f_0^{NS}(h_j), \qquad & j> i,\\
 0  \qquad & j< i .
\end{cases}
$$
We can thus solve the system $X'(t)=AX+B$ starting from the last row up to the first one.
Since the last (resp. first) row corresponds to the highest (resp. lowest) hierarchy level,
opinion propagates from the highest hierarchy level $h_N$ down to the lowest $h_1$, in agreement with our intuition.

More precisely we have

\begin{thm}\label{Thmp0} For any $i=1,..,N$, 
$$ \lim_{t\to +\infty}f^{NS,i}(t) = \delta_{m_\infty^{NS,i}} $$
where
\begin{equation}\label{minfinito9}
 m_\infty^{NS,N} :=
\begin{cases}
m_0^N \quad & \text{if $f_0^S(h_N)=0$,}\\
m_0^{S,N} \quad & \text{if $f_0^S(h_N)>0$},
\end{cases}
\end{equation}
and for $i=1,...,N-1$,
\begin{equation}\label{minfinito10}
m_\infty^{NS,i}=
 \frac{(1-\alpha)\sum_{j\ge i+1}  f_0^{NS}(h_j) m_\infty^{NS,j}
+\alpha \sum_{j\ge i} f_0^S(h_j) m_0^{S,j} }
{\alpha f_0^S\left([h_i,1]\right) + (1-\alpha)f_{0}^{NS}\left([h_{i+1},1]\right)}.
\end{equation}

In particular if there is no stubborn individual in the subpopulation of hierarchy level $h_1,...,h_{N-1}$
then $m_\infty^{NS,i}=m^{NS,N}_\infty$ for $i=1,...,N-1$.
\end{thm}

\begin{proof}
Notice that
$$\frac{d}{dt} m^{NS,N}(t) = A_{NN}m^{NS,N}(t) + B_N $$
with
$$ A_{NN}=(1-\alpha)f_0^{NS}(h_N)-f_0(h_N) = -\alpha f_0^S(h_N),
\quad
B_N =  \alpha m_0^{S,N} f_0^S(h_N) .$$

Let us first examine the case $f_0^S(h_N)=0$ i.e. there are no stubborn agents with hierarchy $h_N$.
Then $A_{NN}=B_N=0$ so that $\frac{d}{dt} m_t^{NS,N}=0$ i.e.
$m^N(t)=m^{NS,N}(t)=m_0^N$. It follows that $f^{NS,N}(t)\to \delta_{m_\infty^{NS,N}}$ with $m_\infty^{NS,N}:=m_0^N$.
If on the other hand $f_0^S(h_N)>0$, so that $\alpha>0$, then
$$ m^{NS,N}(t) \to -B_N/A_{NN}= m_0^{S,N}=:m_\infty^{NS,N}. $$
We thus obtain \eqref{minfinito9}.

We can now determine iteratively the limit of $m^{NS,i}(t)$ for $i=N-1,N-2,...,1$.
Indeed  from
$$ \frac{d}{dt} m^{NS,i}(t) = A_{ii}m^{NS,i}(t) + \sum_{j>i} A_{ij} m^{NS,j}(t) + B_i $$
we see that
\begin{equation*}\label{TheoreaticalLimitp0}
 \lim_{t\to +\infty}m^{NS,i}(t) = m_\infty^{NS,i}:= \frac{\sum_{j>i} A_{ij} m_\infty^{NS,j} + B_i}{-A_{ii}} .
\end{equation*}
Since
\begin{eqnarray*}
 -A_{ii} & = & f_{0}\left([h_i,1]\right)-(1-\alpha)f_0^{NS}(h_i) \\
& = & \alpha f_0^S\left([h_i,1]\right) + (1-\alpha)f_{0}^{NS}\left([h_{i+1},1]\right)
\end{eqnarray*}
we deduce \eqref{minfinito10}.

Denote $\alpha_j$  the proportion of stubborn individuals within the $h_j$-subgroup.
 Observe that $\alpha_j f_0(h_j)=\alpha f_0^S(h_j)$ is the proportion of stubborn individuals with hierarchy $h_j$. 
Likewise $(1-\alpha_j) f_0(h_j)=(1-\alpha) f_0^{NS}(h_j)$ is the proportion of non-stubborn individuals with hierarchy $h_j$.
Then the limit mean opinion in the $h_j$-subgroup is $m_\infty^j = \alpha_j m_0^{S,j} + (1-\alpha_j)m_\infty^{NS,j}$.
So, \eqref{minfinito10} can also be written as
$$ m_\infty^{NS,i}
= \frac{\sum_{j\ge i+1} f_0(h_j)m_\infty^j +  \alpha_i f_0(h_i)m_0^{S,i} }
       {\alpha_i f_0(h_i)+f_0([h_{i+1},1])}
\qquad i=1,..,N-1.
$$

In particular if there is no stubborn agent in the population $h_1,...,h_{N-1}$
i.e. $\alpha_1=...=\alpha_{N-1}=0$ then
$$ m_\infty^{NS,i}
= \frac{\sum_{j\ge i+1} f_0(h_j)m_\infty^j }{f_0([h_{i+1},1])}
\qquad i=1,..,N-1.
$$
We can then easily prove by induction that $m_\infty^i=m^N_\infty$ for $i=1,...,N-1$.
Thus in that case the limit opinion $m_\infty^N$ of the agents with highest hierarchy $h_N$ spreads to all the non-stubborn agents of lower hierarchy.

\end{proof}

The previous Theorem shows that when $p=0$, the information flows from top-to-bottom being only modified by  stubborn individuals. 
When $p\in (0,1)$, this vertical transmission is disrupted. We examine this case in the next subsection. 

\subsubsection{Analysis of the case $p\in (0,1)$.}

We begin our analysis of the case $p\in (0,1)$ showing that when there is a positive fraction $\alpha$ of stubborn individuals in the population, 
the eigenvalues of $A$ have negative real parts. 

\begin{prop}\label{Gershgorin}
The eigenvalues of $A$ belong to the set  $\{z\in\mathbb{C}:\,Re\,z\le -p\alpha - (1-p)\alpha f_0^S(h_N)\}$.
\end{prop}

\begin{proof}
According to Gershgorin's Discs Theorem, the eigenvalues of $A$ lies in the union of the discs
$D(A_{ii},\sum_{j\neq i}|A_{ij})$, $i=1,..,N$.
Let us verify that any such disc lies in $\{Re\,z\le -p\alpha - (1-p)\alpha f_0^S(h_N)\}$.
Notice that $A_{ij}\ge 0$, $j\neq i$. Concerning $A_{ii}$, notice that it is non-increasing in $p$ (because $\frac{d}{dp}A_{ii}=f_0([h_i,1])-1\le 0$) and if $p=0$ we have that
$A_{ii}=(1-\alpha)f_0^{NS}(h_i)-f_{0}\left([h_i,1]\right)\le 0$. Thus $A_{ii}\le 0$.
Writing
$$f_{0}\left([h,1]\right) = \alpha f_{0}^S\left([h,1]\right)+(1-\alpha)f_{0}^{NS}\left([h,1]\right)$$
and
$$ \sum_{j\neq i} A_{ij}=\sum_{j<i}A_{ij}+\sum_{j>i}A_{ij}
= (1-\alpha)p f_0^{NS}([0,h_{i-1}])+(1-\alpha)f_0^{NS}([h_{i+1},1]),$$
we have
\begin{align*}
 &A_{ii} + \sum_{i\neq j} |A_{ij}|= A_{ii} + \sum_{i\neq j} A_{ij} \\
 &=(1-\alpha)f_0^{NS}(h_i)-(p+(1-p)f_{0}\left([h_i,1]\right))\\
 &\quad +(1-\alpha)p f_0^{NS}([0,h_{i-1}])+(1-\alpha)f_0^{NS}([h_{i+1},1])\\
 &= (1-\alpha)f_0^{NS}(h_i)-p -(1-p)\alpha f_{0}^S\left([h_i,1]\right)
-(1-p)(1-\alpha) f_{0}^{NS}\left([h_i,1]\right) \\
&\quad  +(1-\alpha)p f_0^{NS}([0,h_{i-1}])+(1-\alpha)f_0^{NS}([h_{i+1},1]).
\end{align*}
Since $f_0^{NS}(h_i)+f_0^{NS}([h_{i+1},1])=f_0^{NS}([h_i,1])$.\\
Thus
\begin{align*} &A_{ii} + \sum_{i\neq j} |A_{ij}|\\
&=p(1-\alpha)f_0^{NS}([h_{i},1])-p-(1-p)\alpha f_{0}^{S}[h_i,1]+(1-\alpha)p f_0^{NS}([0,h_{i-1}])\\
&= -p\alpha - (1-p)\alpha f_{0}^S\left([h_i,1]\right).\end{align*}
Thus the eigenvalue of $A$ belongs to $\{Re\,z\le -p\alpha - (1-p)\alpha f_0^S(h_N)\}$.
\end{proof}

As a consequence if $\alpha>0$  the eigenvalues of $A$ are  negative.
Then $\lim_{t\to +\infty}X(t) = X_\infty$ with $AX_\infty+B=0$ i.e. $X_\infty=-A^{-1}B$.
We thus obtain the following result:

\begin{thm}\label{ThmLimit20} Suppose that $p>0$ and $\alpha>0$.
Then
\begin{equation}\label{Limit20}
 \lim_{t\to +\infty}f^{NS,i}(t) = \delta_{m_\infty^{NS,i}} \qquad i=1,..,N,
\end{equation}
where $(m_\infty^{NS,1},...,m_\infty^{NS,N})^T = -A^{-1}B$.
\end{thm}

We show in the next section numerical simulations of the agent model which agree completely with the conclusion of Theorems  
\ref{Thmp0}  and \ref{ThmLimit20}. 
We will also observe that when $\alpha=0$, individuals also share  asymptotically the same opinion in each hierarchy level. 
Unfortunately we were not able to deal theoretically with this case $\alpha=0$. 
It usually requires an explicit conserved quantity (see e.g. \cite{pps}) we could not find.

\section{Computational simulations}
 
We present in this section some simulations of  the agent model and compare the numerical limit opinions in each hierarchy level with the theoretical predictions \eqref{minfinito9}-\eqref{minfinito10} and \eqref{Limit20} found above.

The simulations of the agent model were done considering a population of $N=10000$ individuals divided in three hierarchy levels
$h_1=0$, $h_2=0.4$, and $h_3=1$ in respective proportion of $f_0(h_1)=0.2$, $f_0(h_2)=0.7$, and $f_0(h_3)=0.1$.
Initially opinion are distributed uniformly at random in each hierarchy level as $Unif(-0.9,-0.7)$ in $h_1$,
$Unif(-0.5,0.5)$ in $h_2$, $Unif(0.8,1)$ in $h_3$.
The strength $\gamma$ of the attraction in the interaction rule \eqref{InterRules} was taken as $\gamma=0.01$.

We run the simulation with or without the presence of stubborn individuals considering in each case the following values of $p$:
$p=0$, $p=0.25$, $p=0.5$ and $p=1$.

We show in Figure 1 below the evolution in each hierarchy level of the mean opinion in the non-stubborn population together with the evolution of its variance in inset. The theoretical limit opinion is indicated by black dashed line.
Figure 2 below shows the evolution of the distribution of opinion in the whole population.
In both Figures each row corresponds to a value of $p$ (from Top to Bottom, $p=0,\,0.25,\,0.5,\,1$).
The left column corresponds to simulations without stubborn agents, and the right column to simulations with stubborn agents.
In that case we  suppose that there is a proportion of stubborn individual in the
hierarchy levels $h_1$, $h_2$, $h_3$ equal to $0.5$, $0.3$, $0.8$ respectively.

We can observe that the opinion distribution in each hierarchy level converges to a Dirac mass.
In the cases $p=0$ and $p>0$, $\alpha>0$, the limit opinion values agree perfectly with the theoretical predictions
\eqref{minfinito9}-\eqref{minfinito10} and \eqref{Limit20}  indicated by black dashed line.
Comparing the top row for $p=0$ with the 2nd row $p=0.25$, 
we can appreciate  how the noise $p$ disrupts the top-to-bottom transmission of information. 
The variance shown in the inset converge to 0 exponentially fast in agreement with Prop. \ref{Prop_Mean_Mixed}.
We can also see that the lower ranked population reach consensus earlier than the higher ranked one.
This also agrees with the estimate of the velocity of convergence given in Prop. \ref{Prop_Mean_Mixed}.
Eventually this convergence seems  to hold also in the case $p>0$, $\alpha=0$ which we could not treat theoretically.

\begin{figure}[H]\label{Fig_p0}
\medskip
\begin{minipage}{.49\textwidth}
  \centering
  \includegraphics[width=0.95\linewidth]{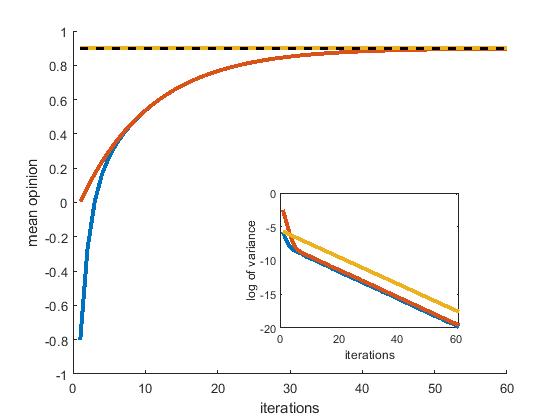}
\end{minipage}
\begin{minipage}{.49\textwidth}
  \centering
  \includegraphics[width=0.95\linewidth]{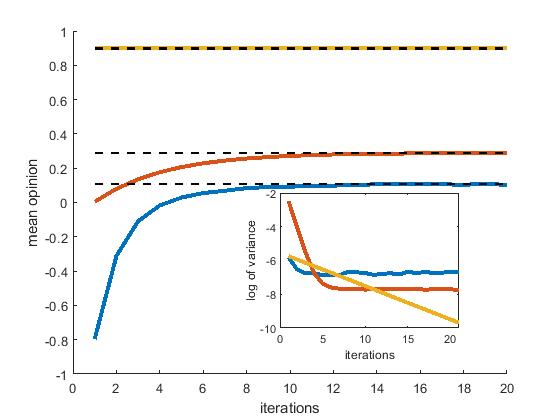}
\end{minipage}
\begin{minipage}{.49\textwidth}
  \centering
  \includegraphics[width=0.95\linewidth]{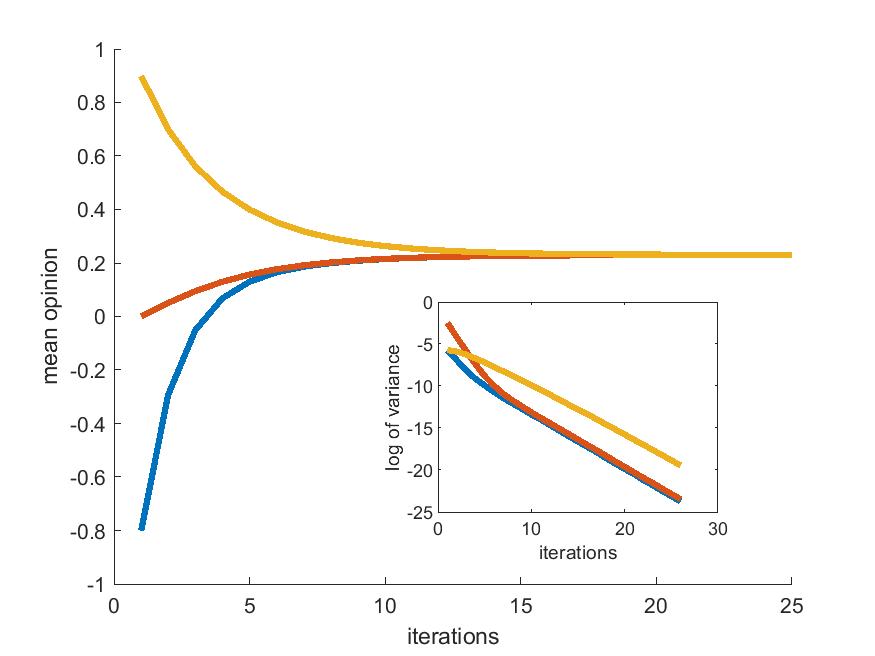}
\end{minipage}
\begin{minipage}{.49\textwidth}
  \centering
  \includegraphics[width=0.95\linewidth]{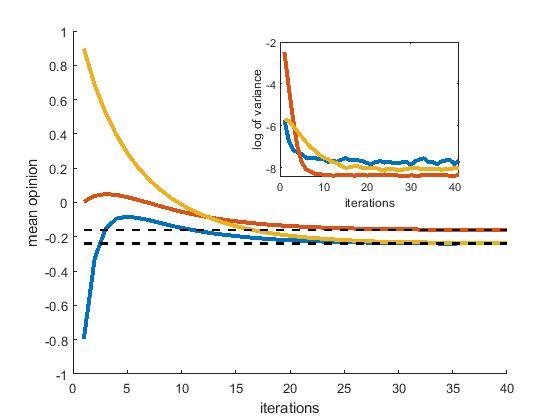}
\end{minipage}
\begin{minipage}{.49\textwidth}
  \centering
  \includegraphics[width=0.95\linewidth]{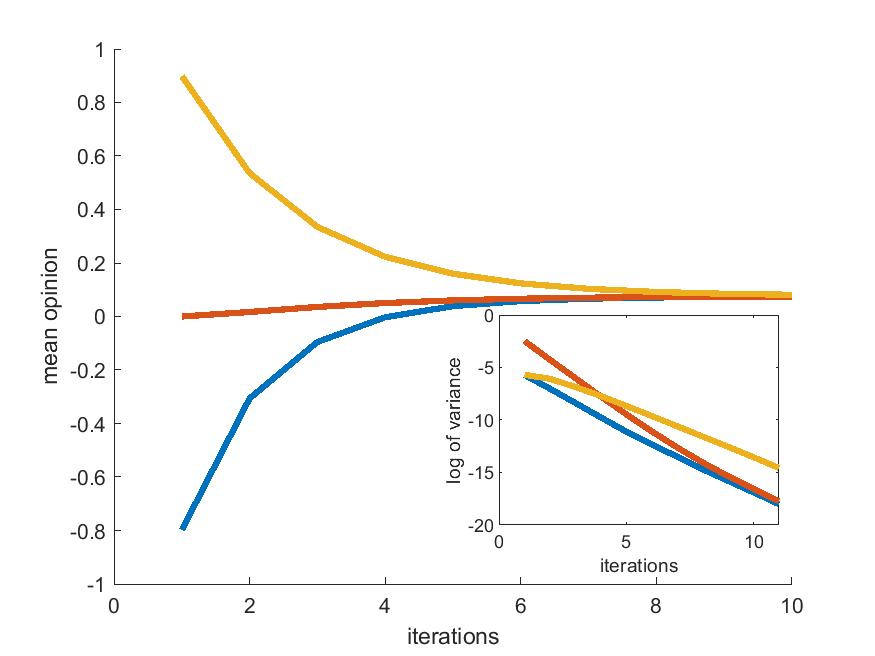}
\end{minipage}
\begin{minipage}{.49\textwidth}
  \centering
  \includegraphics[width=0.95\linewidth]{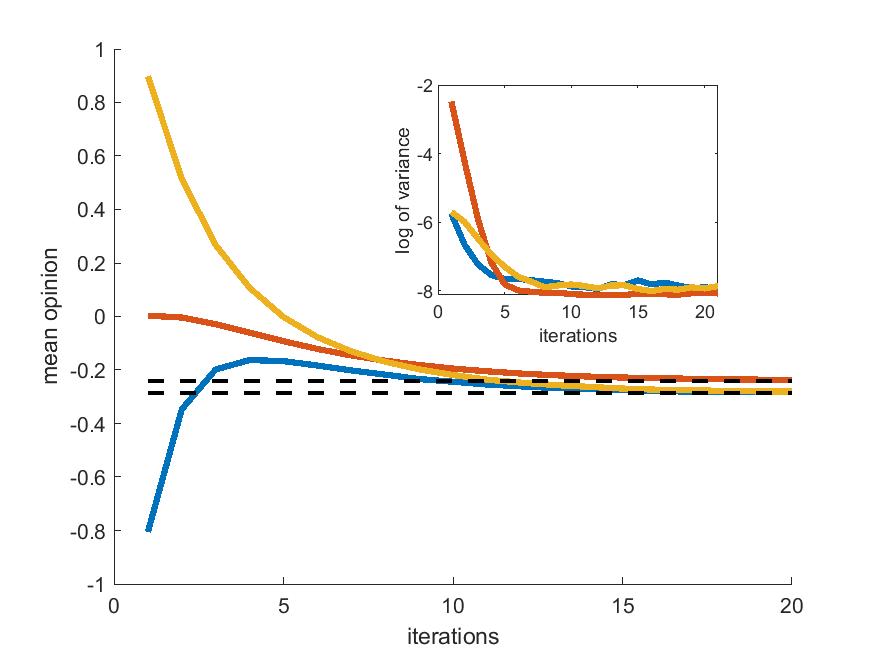}
\end{minipage}
\begin{minipage}{.49\textwidth}
  \centering
  \includegraphics[width=0.95\linewidth]{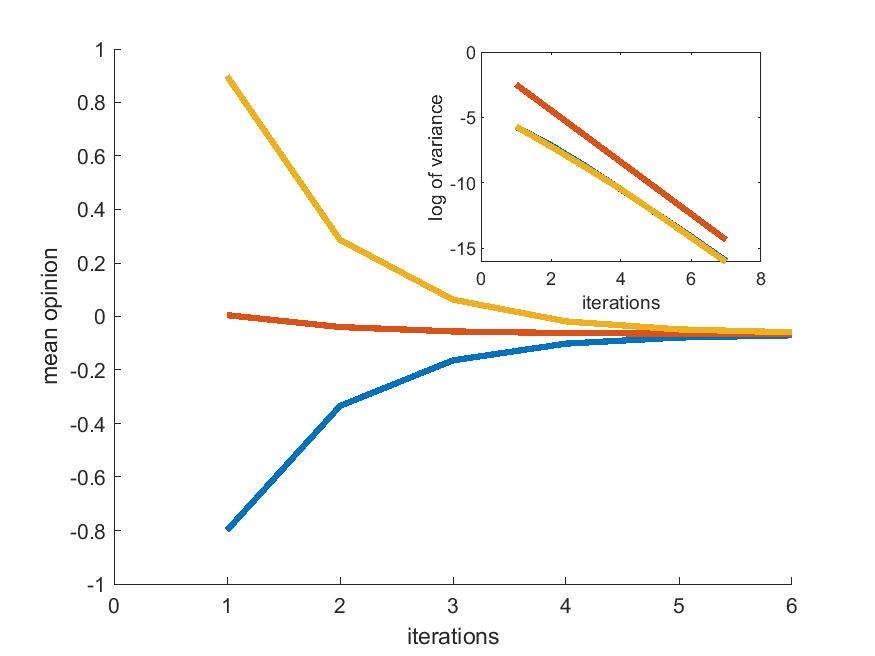}
\end{minipage}
\begin{minipage}{.49\textwidth}
  \centering
  \includegraphics[width=0.95\linewidth]{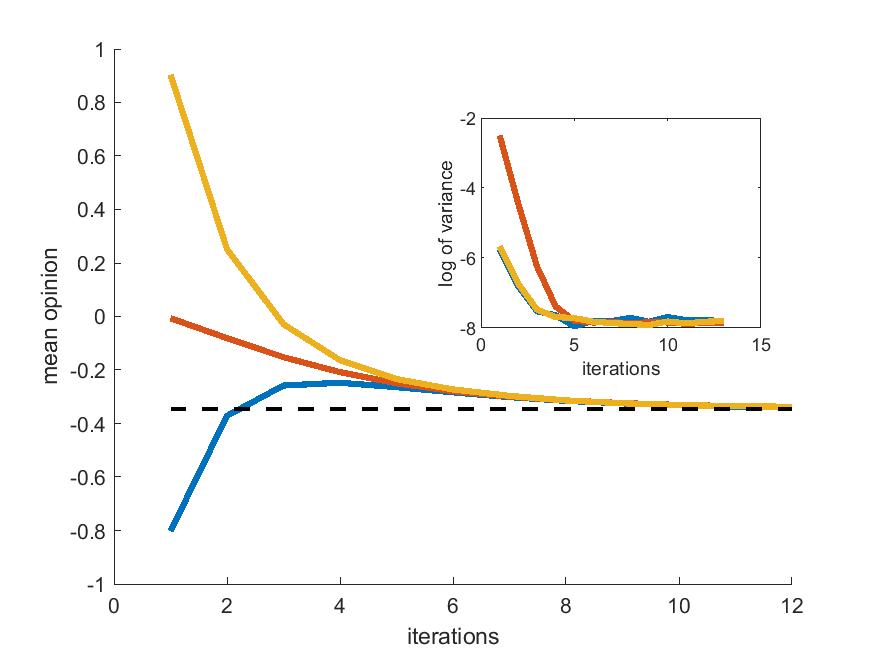}
\end{minipage}
\caption{Time evolution of the mean opinion and its variance (inset) in the non-stubborn population for each hierarchy level with/without stubborn and for different values of $p$. From Top to Bottom: $p=0$, $p=0.25$, $p=0.75$, $p=1$. Left: without stubborn agent, Right: with stubborn agent.
The theoretical limit opinions (available when $p=0$ and $p>0$ with stubborn agents) are indicated by black dashed horizontal lines. }
\end{figure}

\begin{figure}[H]\label{Fig_p0}
\medskip
\begin{minipage}{.49\textwidth}
  \centering
  \includegraphics[width=0.95\linewidth]{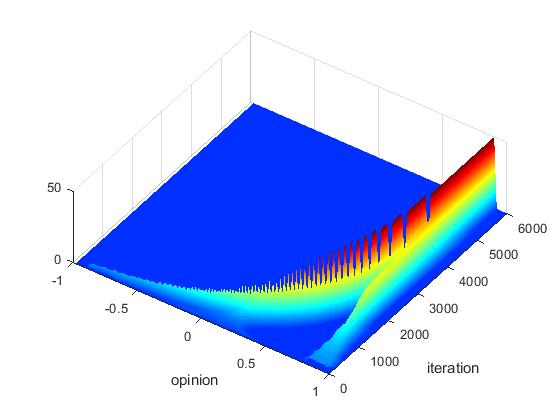}
\end{minipage}
\begin{minipage}{.49\textwidth}
  \centering
  \includegraphics[width=0.95\linewidth]{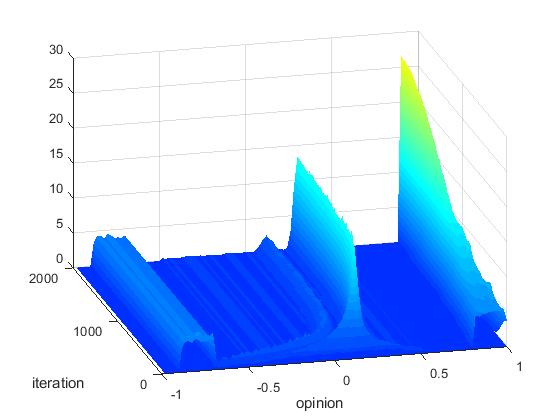}
\end{minipage}
\begin{minipage}{.49\textwidth}
  \centering
  \includegraphics[width=0.95\linewidth]{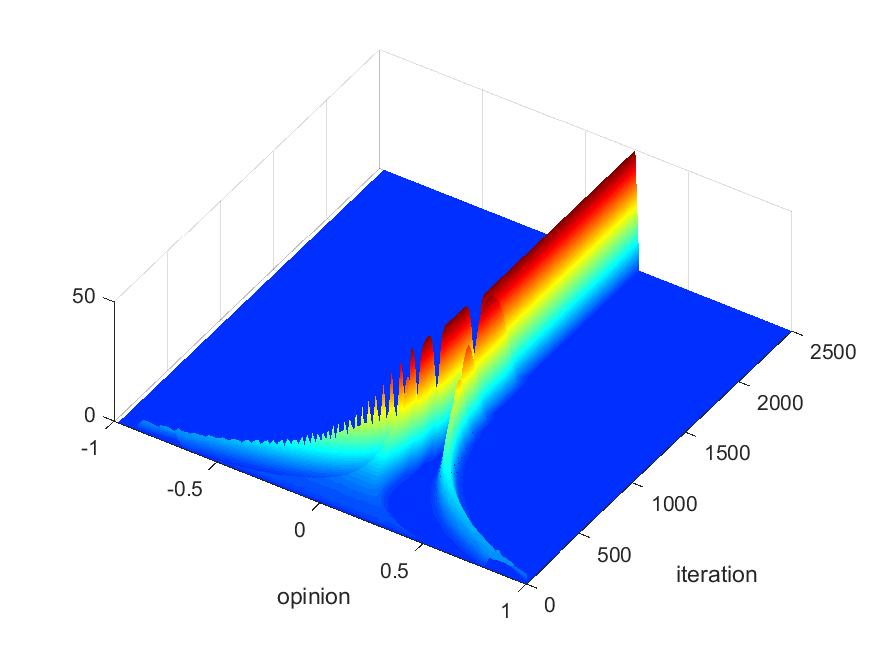}
\end{minipage}
\begin{minipage}{.49\textwidth}
  \centering
  \includegraphics[width=0.95\linewidth]{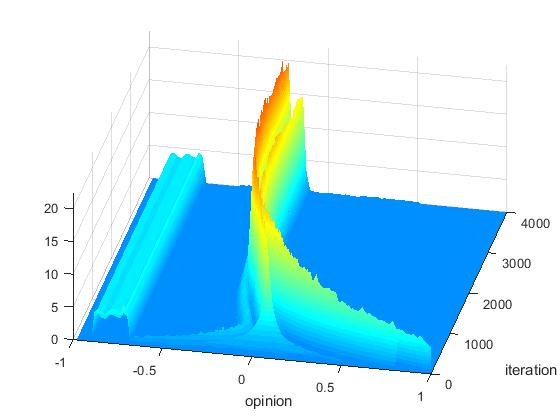}
\end{minipage}
\begin{minipage}{.49\textwidth}
  \centering
  \includegraphics[width=0.95\linewidth]{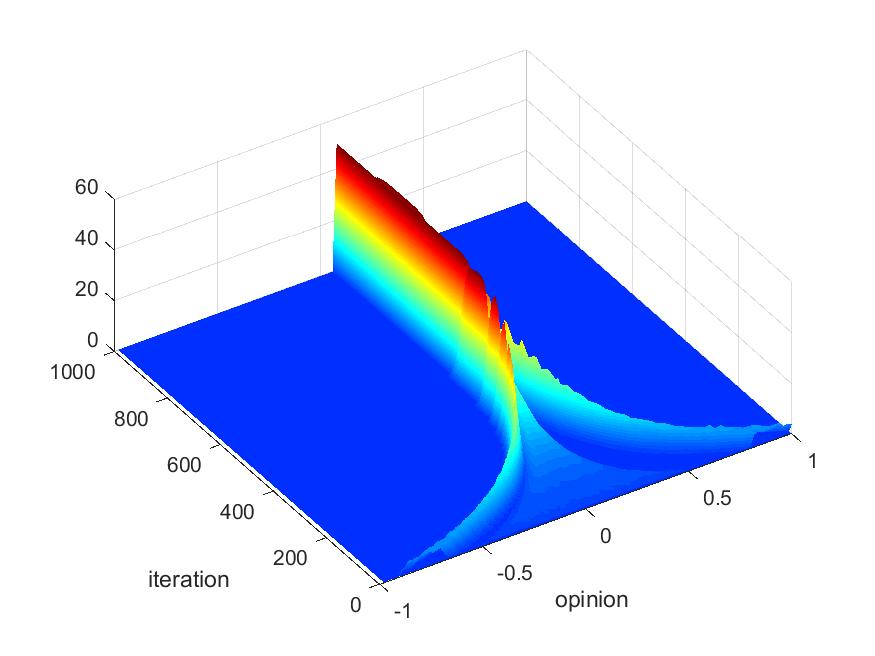}
\end{minipage}
\begin{minipage}{.49\textwidth}
  \centering
  \includegraphics[width=0.95\linewidth]{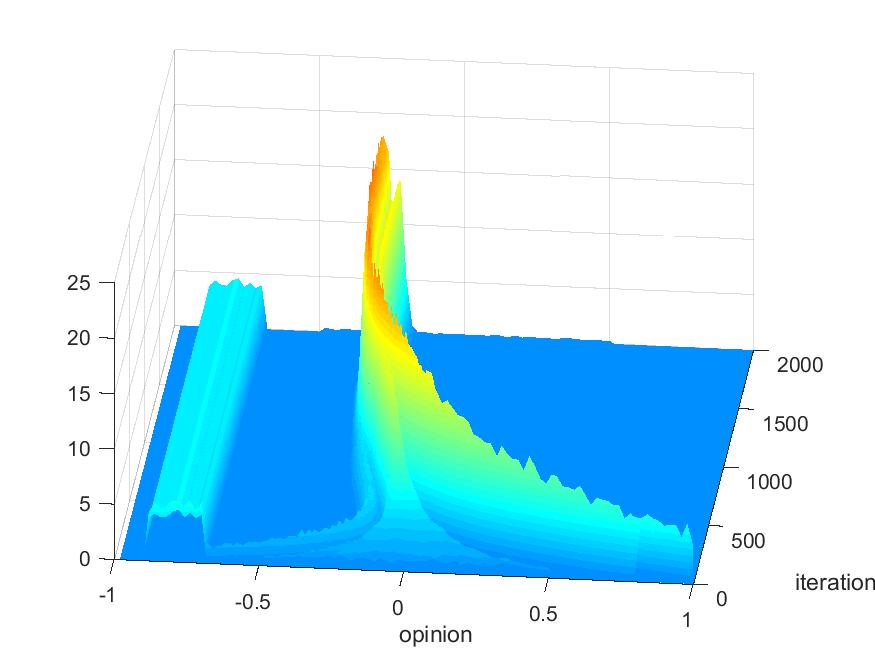}
\end{minipage}
\begin{minipage}{.49\textwidth}
  \centering
  \includegraphics[width=0.95\linewidth]{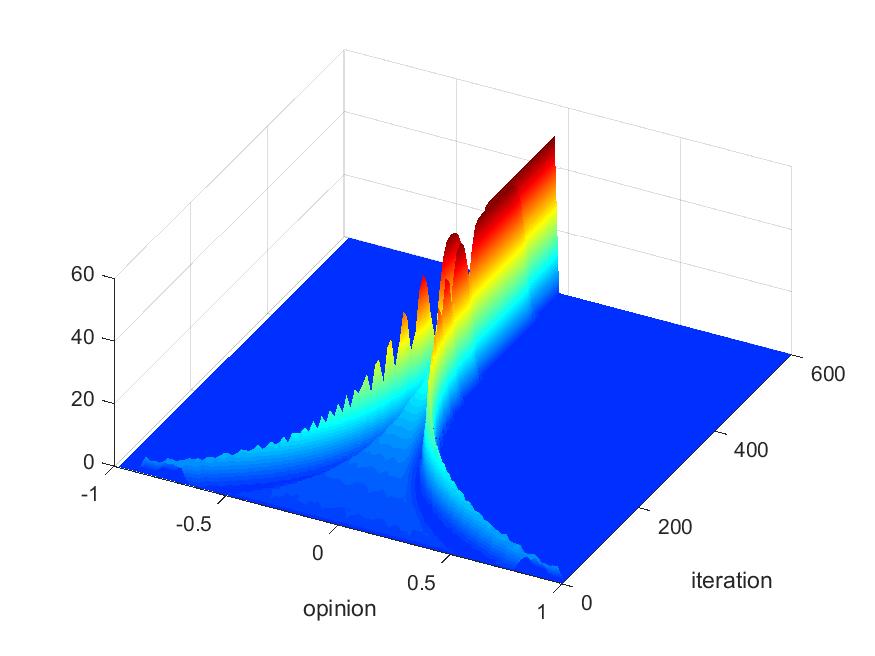}
\end{minipage}
\begin{minipage}{.49\textwidth}
  \centering
  \includegraphics[width=0.95\linewidth]{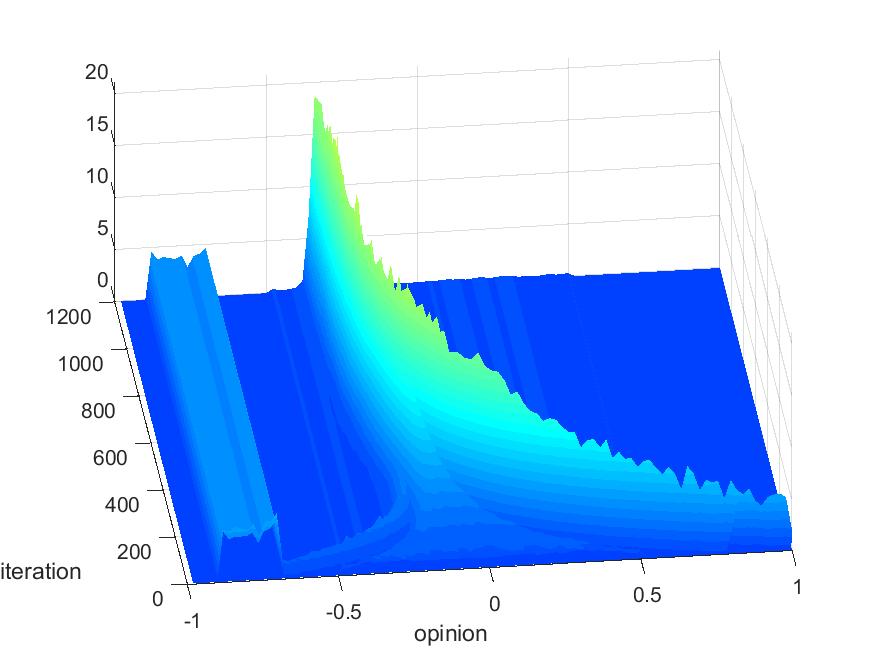}
\end{minipage}
\caption{Time evolution of the distribution of opinion in the whole population for the same model parameters as in Figure 1. 
From Top to Bottom: $p=0$, $p=0.25$, $p=0.75$, $p=1$. Left: without stubborn agent, Right: with stubborn agent.
}
\end{figure}

\end{document}